\documentclass[a4paper,12pt]{article}

\usepackage{amscd}
\usepackage{amsfonts}
\usepackage{amsmath}
\usepackage{amssymb}
\usepackage{amsthm}
\usepackage{ascmac}
\usepackage{here}
\usepackage{makeidx}
\usepackage{mathrsfs}
\usepackage{slashbox}
\usepackage{txfonts}

\allowdisplaybreaks
\makeindex

\newcommand{\C}{\mathbb{C}}
\newcommand{\F}{\mathbb{F}}
\newcommand{\N}{\mathbb{N}}
\newcommand{\Q}{\mathbb{Q}}

\newcommand{\Z}{\mathbb{Z}}
\newcommand{\A}{\mathscr{A}}

\newcommand{\Fil}{\mathscr{F}}

\newcommand{\M}{\mathscr{M}}

\newtheorem{thm}{Theorem}[section]
\newtheorem{dfn}[thm]{Definition}

\newtheorem{prp}[thm]{Proposition}
\newtheorem{crl}[thm]{Corollary}
\newtheorem{rmk}[thm]{Remark}
\newtheorem{exm}[thm]{Example}

\def\ens#1{\{ #1 \}}
\def\Ens#1{\left\{ \ #1 \ \right\}}
\def\set#1#2{\{ \ #1 \mid #2 \ \}}
\def\Set#1#2{\left\{ \ #1 \ \middle|\ #2 \ \right\}}
\def\m#1{{\rm #1}}
\def\v#1{| #1 |}
\def\n#1{\| #1 \|}
\newcommand{\im}{\m{im}}

\title{Semisimplicity and Reduction of $p$-adic Representations of Topological Monoids}
\author{Tomoki Mihara}
\date{}

\begin{document}

\maketitle
\begin{abstract}
We give a criterion of the semisimplicity of a $p$-adic unitary representation of a topological monoid by the reduction of the associated operator algebra.
\end{abstract}

\tableofcontents

\section{Introduction}
\label{Introduction}

Let $k$ be a local field. The reduction of a unitary representation of a topological monoid $\M$ over $k$ does not preserve the irreducibility. It is because the reduction only reflects the action of the integral model $k^{\circ}[\M] \subset k[\M]$. We verify that the reduction with respect to a larger integral model compatible with the operator norm preserves the simplicity of a left module in Theorem \ref{simplicity}. This is extended to the reduction of an operator algebra associated to a semisimple unitary representation in Theorem \ref{semisimplicity} by the lifting of central idempotents in Corollary \ref{lift 2}, and gives a criterion of the semisimplicity of a $p$-adic unitary representation of a topological module in Theorem \ref{representation}.

This theory is a generalisation of the reduction theory of the spectrum of an operator in \cite{Mih}. The most essential technique of the reduction theory in \cite{Mih} is a repetition of reductions of an operator. A similar technique is also essential in this paper for the calculation of the reduction of an operator algebra with respect to the suitable integral model. We deal with the repetition of reductions in \S \ref{Semisimplicity of a Unitary Representation} and \S \ref{Examples}.

For a profinite group $G$, this theory connects the reduction of unitary representations of $G$ and the reduction of the $p$-adic unitary dual $\check{G}$ of $G$. In particular when $\M = \Z_p$, then its unitary dual is the open unit ball in $\C_p$ centred at $1$ by Amice's theory, and the connection between two reductions corresponds to the compatibility of the reduction and the Fourier transform.

\vspace{0.1in}

We recall basic notions of $p$-adic Banach algebras and $p$-adic unitary representations in \S \ref{Preliminaries}. In order to observe a relation between the semisimplicity and the reduction of Banach modules, we introduce the lifting properties of idempotents and decompositions in \S \ref{Decomposition of Rings}. We apply the results of \S \ref{Decomposition of Rings} to $p$-adic unitary representations of a topological module in \S \ref{Connection to Representation Theory}. Finally we observe the relation between our theory and the $p$-adic unitary dual of the profinite group together with an example on Amice's theory of Fourier transform $\Z_p$ in \S \ref{$p$-adic Unitary Dual}.

\section{Preliminaries}
\label{Preliminaries}

We recall basic notions of $p$-adic Banach algebras and $p$-adic unitary representations. Here ``unitary'' means that the action preserves the integral structure give as the unita ball. In particular a unitary representation of a group is isometric, and then there is no ambiguity. However, we also deal with a unitary operator of a monoid, and it is just submetric in general.

\subsection{Banach Algebra}
\label{Banach Algebra}

Let $A$ be a ring. A ring is assumed to be associative and unital, but not necessarily to be commutative. A ring homomorphism is assumed to be unital. For an $S \subset A$, we denote by $(S,A)' \subset A$ the subring of elements $t$ with $st = ts$ for any $s \in S$. If there is no ambiguity of $A$, then we simply put $S' \coloneqq (S,A)'$. A $c \in A$ is said to be {\it central} in $A$ if $c \in A'$. For a  commutative ring $R$, an {\it $R$-algebra} is a ring $A$ endowed with a ring homomorphism $R \to A$ whose image lies in $A'$.

Let $k$ be a valuation field. We do not assume that the valuation is non-trivial. We always fix a non-Archimedean norm $\v{\cdot} \colon k \to [ 0, \infty )$ associated to the valuation of $k$, and regard $k$ as a topological field with respect to the induced ultrametric. A {\it normed $k$-vector space} is a $k$-vector space $V$ endowed with a non-Archimedean norm $\n{\cdot} \colon V \to [ 0, \infty )$ with $\n{av} = \v{a} \ \n{v}$ for any $(a,v) \in k \times V$. For a normed $k$-vector space $V$, we set $V(1) \coloneqq \set{v \in V}{\n{v} \leq 1}$ and $V(1-) \coloneqq \set{v \in V}{\n{v} < 1}$. In particular, we put $k^{\circ} \coloneqq k(1)$ and $k^{\circ \circ} \coloneqq k(1-)$. Then $V(1-) \subset V(1) \subset V$ are $k^{\circ}$-submodules of $V$. We denote by $\overline{V}$ the quotient $k^{\circ}$-module $V(1)/V(1-)$. Since $k^{\circ \circ} \subset k^{\circ}$ is a unique maximal ideal, $\overline{k}$ is a field. The action of $k^{\circ \circ}$ on $\overline{V}$ is trivial, and hence $\overline{V}$ is a $\overline{k}$-vector space.

Let $k$ be a complete valuation field. A {\it Banach $k$-vector space} is a normed $k$-vector space complete with respect to the ultrametric induced by the norm. For a Banach $k$-vector space $V$, we denote by $\m{End}_k(V)$ the $k$-algebra of $k$-linear endomorphisms of the underlying $k$-vector space of $V$. The strong topology of $\m{End}_k(V)$ is the locally convex topology of pointwise convergence.

A {\it Banach $k$-algebra} is a $k$-algebra $\A$ endowed with a norm $\n{\cdot} \colon \A \to [ 0,\infty )$ satisfying the following:
\begin{itemize}
\item[(i)] The underlying $k$-vector space of $\A$ endowed with $\n{\cdot}$ is is a Banach $k$-vector space. 
\item[(ii)] $\n{ab} \leq \n{a} \ \n{b}$ for any $(a,b) \in \A^2$.
\item[(iii)] $\n{1} \in \Ens{0,1}$.
\item[(iv)] $\n{ca} = \v{c} \ \n{a}$ for any $(c,a) \in k \times \A$.
\end{itemize}
We also denote by $\A$ the underlying Banach $k$-vector space. Then $\A(1-) \subset \A(1) \subset \A$ are $k^{\circ}$-subalgebras, and $\overline{\A}$ is a $\overline{k}$-algebra. For example, for a Banach $k$-vector space $V$, the $k$-subalgebra $\m{B}_k(V) \subset \m{End}_k(V)$ of continuous $k$-linear endomorphisms is a Banach $k$-algebra with respect to the operator norm given in the following way:
\begin{eqnarray*}
  \n{\cdot} \colon \m{B}_k(V) & \to & [ 0, \infty ) \\
  T & \mapsto & \sup_{v \in V(1)} \n{Tv} < \infty.
\end{eqnarray*}
For a Banach $k$-algebra $\A$, a {\it Banach left $\A$-module} is a left $\A$-module $M$ endowed with a complete non-Archimedean norm $\n{\cdot} \colon M \to [ 0,\infty )$ of the underlying $k$-vector space with $\n{am} \leq \n{a} \ \n{m}$ for any $(a,m) \in \A \times M$.

A {\it local field} is a complete discrete valuation field $k$ with finite residue field $\overline{k}$. We denote by $p > 0$ the characteristic of $\overline{k}$.

\subsection{Unitary Representation of a Topological Monoid}
\label{Unitary Representation of a Topological Monoid}

A {\it topological monoid} is a monoid $\M$ endowed with a topology with respect to which the multiplication $\M \times \M \to \M$ is continuous. In this subsection, let $k$ be a complete valuation field, and $\M$ a topological monoid.

\begin{dfn}
A unitary representation of $\M$ over $k$ is a pair $(V,\rho)$ of a Banach $k$-vector space $V$ and a monoid homomorphism $\rho \colon \M \to \m{End}_k(V)$ with respect to the multiplication of $\m{End}_k(V)$ such that $\n{\rho(m)(v)} \leq \n{v}$ for any $(m,v) \in \M \times V$ and the associated action
\begin{eqnarray*}
  \tilde{\rho} \colon \M \times V & \to & V \\
  (m,v) & \mapsto & \rho(m)(v)
\end{eqnarray*}
is continuous. A strictly Cartesian unitary representation of $\M$ over $k$ is a unitary representation $(V,\rho)$ of $\M$ over $k$ with $\n{V} \subset \v{k}$.
\end{dfn}

If $k$ is a complete discrete valuation field or if $V$ is of countable type, then the condition $\n{V} \subset \v{k}$ guarantees the existence of an orthonormal Schauder basis of $V$. This is why we use the term ``strictly Cartesian''.

The multiplicative submonoid $\m{B}_k(V)(1) \subset \m{End}_k(V)$ of submetric $k$-linear endomorphisms is equicontinuous by Banach--Steinhaus theorem (\cite{Sch02} Corollary 6.16), and hence the continuity of the action $\tilde{\rho}$ is equivalent to the continuity of $\rho$ with respect to the strong topology of $\m{End}_k(V)$.

\begin{dfn}
Let $(V,\rho)$ and $(W,\pi)$ be unitary representations of $\M$ over $k$. We say that $(V,\rho)$ is isomorphic to $(W,\pi)$ if there is a homeomorphic $k$-linear isomorphism $V \to W$ preserving the action of $\M$.
\end{dfn}

This relation is an equivalent relation. Beware that we do not assume that the isomorphism $V \to W$ is an isometry, and hence a replacement of the norm by an equivalent norm with respect to which the action of $\M$ is unitary gives an isomorphism. In particular, there is a one-to-one correspondence between the set of isomorphism classes of finite dimensional strictly Cartesian unitary representations of $\M$ over $k$ and the set of isomorphism classes of finite dimensional continuous representations $(V,\rho)$ of $\M$ over $k$ which are unitarisable by a norm $\n{\cdot} \colon V \to k$ with $\n{V} \subset \v{k}$, and hence it can regarded as a subset of the set of isomorphism classes of finite dimensional continuous representations $(V,\rho)$ of $\M$ over $k$. This identification relies on the fact that a Hausdorff locally convex topology of a finite dimensional $k$-vector space is unique and a norm of it is unique up to isomorphisms.

\section{Decomposition of Rings}
\label{Decomposition of Rings}

In this section, let $k$ denote a local field. We deal with the relation between a decomposition of a ring by two-sided ideals and the reduction. A decomposition of a ring is given by a central idempotent. We observe the lifting properties of idempotents first, and after then we prove the compatibility of the semisimplicity and the reduction.

\subsection{Lifting of Idempotents}
\label{Lifting of Idempotents}

For a ring $A$, an $e \in A$ is said to be an {\it idempotent} if $e^2 = e$. We verify lifting properties of idempotents for the reduction of Banach algebras. This is a generalisation of the lifting property of (central) idempotents for the projection $k^{\circ}[G] \to \overline{k}[G]$ for a finite group $G$. Since the commutant is not compatible with the reduction in general, one needs to calculate the commutator to lift a central idempotent.

\begin{prp}
\label{lift 1}
Let $\A$ be a Banach $k$-algebra with $\n{\A} \subset \v{k}$. For any idempotent $\overline{e} \in \overline{\A}$, there is an idempotent $e \in \A(1)$ such that $e + \A(1-) = \overline{e}$.
\end{prp}

This is the simplest application of \cite{Mih} Proposition 5.8 for an arbitrary lift $P_0 \in \A(1)$ of $\overline{e} \in \overline{\A}$. Since the proof of Proposition 5.8 for a general $P_0$ is a little complicated, we give a shortened proof for this simple case.

\begin{proof}
If $\overline{e} = 0$, then $e \coloneqq 0$ is a desired idempotent. There it suffices to assume $\overline{e} \neq 0$. Take a lift $P_0 \in \A(1)$ of $\overline{e} \in \overline{\A}$. Since $\overline{e}$ is a non-zero idempotent, we have $\n{A} = 1$ and $\n{A^2 - A} < 1$. We define a sequence $(P_i)_{i \in \N} \in \A(1)^{\N}$ inductively by the recurrence relation $P_{i+1} \coloneqq -2P_i^3 + 3P_i^2$. Then for any $i \in \N$,
\begin{eqnarray*}
  & & P_{i+1} - P_i = -2P_i^3 + 3P_i^2 - P_i = (-2P_i+1)(P_i^2 - P_i) \in \A(1-)
\end{eqnarray*}
and
\begin{eqnarray*}
  & & P_{i+1}^2 - P_{i+1} = (-2P_i^3 + 3P_i^2)^2 - (-2P_i^3 + 3P_i^2) = 4P_i^6 - 12P_i^5 + 9P_i^4 + 2P_i^3 - 3P_i^2 \\
  & = & (P_i^2 - P_i)(4P_i^4 - 8P_i^3 + P_i^2 + 3P_i) = (P_i^2 - P_i)(4(P_i^2 - P_i)^2 - 3 (P_i^2 - P_i)) \\
  & = & 4(P_i^2 - P_i)^3 - 3(P_i^2 - P_i)^2.
\end{eqnarray*}
Therefore
\begin{eqnarray*}
  \n{P_{i+1}^2 - P_{i+1}} \leq \n{P_i^2 - P_i}^2 \leq \cdots \leq \n{P_0^2 - P_0}^{2^{i+1}} = \n{A^2 - A}^{2^{i+1}} \stackrel{i \to \infty}{\longrightarrow} 0
\end{eqnarray*}
and
\begin{eqnarray*}
  \n{P_{i+1} - P_{i+1}} \leq \n{(-2P_i+1)} \ \n{(P_i^2 - P_i)} = \n{P_i^2 - P_i} \stackrel{i \to \infty}{\longrightarrow} 0.
\end{eqnarray*}
Thus $(P_i)_{i \in \N}$ converges to a unique idempotent $e \in \A(1)$ with $e - P_0 \in \A(1-)$ because $\A(1)$ is a closed subset of a complete topological ring $\A$. The relation $e - P_0 \in \A(1-)$ implies $e + \A(1-) = P_0 + \A(1-) = \overline{e}$.
\end{proof}

\begin{crl}
\label{lift 2}
Let $\A$ be a Banach $k$-algebra with $\n{\A} \subset \v{k}$, and $O \subset \A(1)$ a closed $k^{\circ}$-subring. For any central idempotent $\overline{e} \in \overline{\A}$, if $\overline{e}$ lies in the image of $O$, then there is an central idempotent $e \in \A$ such that $e \in O \subset \A(1)$ and $e + \A(1-) = \overline{e}$.
\end{crl}

Beware that the inclusion $\overline{\A'} \subset \overline{\A}{}'$ is not an equality in general. Therefore the result can not be obtained by simply applying Proposition \ref{lift 1} to the Banach $k$-algebra $\A'$.

\begin{proof}
Taking an $A \in \A(1)$ as an element of $O$ in the beginning of the proof of Proposition \ref{lift 1}, we obtain an idempotent $e \in O$ with $e + \A(1) = \overline{e}$. Let $a \in \A$. Assume $eae \neq ae$. Since $\n{\A} \subset \v{k}$, there is a $c \in k^{\times}$ such that $c(eae - ae) \in \A(1) \backslash \A(1-)$. Put $b \coloneqq c(eae - ae)$. Since $\overline{e}$ is an idempotent, $eb - be \in \A(1-)$. On the other hand, we have $eb = c(eae - eae) = 0$, $be = c(eae - ae) = b$, and hence $eb - be = -b$. This contradicts $b \notin \A(1-)$. Therefore $eae = ae$. Similarly $eae = ea$. Thus $ae = eae = ea$. We conclude that $e$ is central in $\A$.
\end{proof}

\subsection{Reductively Semisimple Banach Algebras}
\label{Reductively Semisimple Banach Algebras}

Let $R$ be a ring. A left $R$-module $M$ is said to be {\it semisimple} if $M$ is the direct sum of simple submodules. $R$ is said to be a {\it semisimple ring} if its left regular module ${}_{R}R$ is semisimple. We remark that $R$ is semisimple if and only if every left $R$-module is semisimple, and if and only if Jacobson radical of $R$ is trivial. $R$ is said to be a {\it simple ring} if $R$ possesses no non-trivial two-sided ideal. An Artinian simple ring is a semisimple ring by Wedderburn's theorem.

Let $R$ be a ring. A {\it semisimple $R$-algebra} (resp.\ {\it simple $R$-algebra}) is an $R$-algebra whose underlying ring is a semisimple ring (resp.\ a simple ring).

We give a criterion of the simplicity of the underlying $k$-algebra of a Banach $k$-algebra by the reduction.

\begin{thm}
\label{simplicity}
Let $\A$ be a Banach $k$-algebra with $\n{\A} = \v{k}$. If $\overline{\A}$ is a finite dimensional simple $\overline{k}$-algebra, then the underlying $k$-algebra $\A$ is a finite dimensional simple $k$-algebra.
\end{thm}

\begin{proof}
Since $\overline{\A}$ is of finite dimension, $\A(1)$ is a free $k^{\circ}$-module of finite rank. Indeed, let $\overline{a}_1, \ldots \overline{a}_n \in \overline{\A}$ be a $\overline{k}$-basis, and take representatives $a_1, \ldots, a_n \in \A(1)$ of them. For a uniformiser $\varpi \in k^{\circ \circ}$, we have $\A(1-) = \A(1) \varpi$ and hence
\begin{eqnarray*}
  & & \A(1) = k^{\circ} a_1 + \cdots + k^{\circ} a_n + \A(1-) = k^{\circ} a_1 + \cdots + k^{\circ} a_n + \A(1) \varpi \\
  & = & k^{\circ} a_1 + \cdots + k^{\circ} a_n + k^{\circ} a_1 \varpi + \cdots + k^{\circ} a_n \varpi + \A(1-) \varpi = k^{\circ} a_1 + \cdots + k^{\circ} a_n + \A(1) \varpi^2\\
  & = & \cdots = k^{\circ} a_1 + \cdots + k^{\circ} a_n + \A(1) \varpi^N
\end{eqnarray*}
for any $N \in \N$. It implies that $k^{\circ} a_1 + \cdots + k^{\circ} a_m \subset \A(1)$ is a dense $k^{\circ}$-submodule. It is the image of the continuous $k^{\circ}$-linear homomorphism $(k^{\circ})^n \to \A(1)$ associated to $a_1, \ldots, a_n$, and is closed because $(k^{\circ})^n$ is compact and $\A(1)$ is Hausdorff. Therefore $\A(1)$ is generated by $a_1, \ldots, a_n$. Since $\A(1)$ is torsion free, $\A(1)$ is a free $k^{\circ}$-module.

Since $\overline{\A}$ is of finite dimension again, it is Artinian. By Wedderburn's theorem, $\overline{\A}$ is isomorphic to a $\overline{k}$-algebra $\m{M}_l(D)$ of matrices over a division $\overline{k}$-algebra $D$, and through an identification $\overline{\A} \cong_{\overline{\A}-\m{Alg}} \m{M}_l(D)$ every simple left $\overline{\A}$-module is isomorphic to the natural representation $D^l$, where $l \coloneqq \sqrt[]{\mathstrut n} \in \N$. Take a representative $\mu$ of the unique isomorphism class of simple left $\overline{\A}$-modules. Since $\overline{\A}$ is semisimple, every left $\overline{\A}$-module is isomorphic to a direct sum of copies of $\mu$.

By $\n{\A} \subset \v{k}$, the norm of $\A$ coincides with the norm associated to the filtration $\A(1) \supsetneq \varpi \A(1) \supsetneq \varpi^2 \A(1)  \supsetneq \cdots \supsetneq \bigcap_{i \in \N} \varpi^i \A(1) = O$. Since $\A(1)$ is a free $k^{\circ}$-module, $\A$ is strictly Cartesian, and hence every closed $k$-vector subspace of $\A$ is strictly closed by \cite{BGR84} 2.4.2.\ Proposition 1. Since $\A$ is of finite dimension, every $k$-vector subspace of $\A$ is closed, and hence strictly closed.

Let $M \neq O$ be a cyclic left $\A$-module, and $\rho_M \colon \A \to \m{End}_k(M)$ the $k$-algebra homomorphism associated with the action of $\A$ on $M$. Let $a \in \ker(\rho_M)$. Assume $a \neq 0$. Since $\n{\A} \subset \v{k}$, there is a $c \in k^{\times}$ such that $ca \in \A(1) \backslash \A(1-)$. Then $ca \in \ker(\rho_M)$. Since $M$ is cyclic, $M$ is isomorphic to the quotient $\A/I$ by a left ideal $I \subsetneq \A$. We identify $M$ with $\A/I$ and we endow $M$ with the quotient seminorm. By the argument above, $I$ is closed and strictly closed in $\A$. It implies that $M$ is a left Banach $\A$-module, and the identification $M \cong_{\m{Ban}\A-\m{Mod}} \A/I$ induces an isomorphism $\overline{M} \cong_{\overline{\A}-\m{Mod}} \overline{\A}/\overline{I}$. Since $ca \in \A(1)$ acts trivially on $M(1)$, so does $ca + \A(1-) \in \overline{\A}$ on $\overline{M}$. On the other hand, $\overline{M}$ is a direct sum of $\mu$, and hence $ca + \A(1-)$ acts trivially on the unique simple left $\overline{\A}$-module $\mu$. Thus $ca + \A(1-)$ is an element of Jacobson radical of $\overline{\A}$, which is trivial by the assumption of the semisimplicity of $\overline{\A}$. It contradicts the fact $ca \in \A(1) \backslash \A(1-)$, and we obtain $a = 0$. Thus $\rho_M$ is injective.

We identify $\A$ as the underlying $k$-algebra. Since $\A$ is of finite dimension, $\A$ is Artinian and hence admits a simple left module. Every simple left module is cyclic, and hence $\A$ is a primitive ring by the argument above. It implies that a simple left $\A$-module $M$ is unique up to isomorphisms, and $\rho_M(\A)$ is strongly dense in the double commutant $\m{End}_{\m{End}_{\rho_M(\A)}(M)}(M)$ by Jacobson--Bourbaki density theorem (\cite{Cri04} D 2.2). Since $M$ is of finite dimension, the density implies $\rho_M(\A) = \m{End}_{\m{End}_{\rho_M(\A)}(M)}(M)$. It follows from Schur's lemma that $D \coloneqq \m{End}_{\rho_M(\A)}(M)$ is a division $k$-algebra, and hence $\m{End}_D(M)$ is isomorphic to the simple $k$-algebra $\m{M}_{\dim_D M}(D)$. We conclude that $\A$ is a simple $k$-algebra.
\end{proof}

We note that the converse of Theorem \ref{simplicity} does not hold. For example, consider the simple $\Q_2$-algebra $\m{M}_2(\Q_2)$ of finite dimension. It admits the operator norm with respect to the natural module $\Q_2^2$ endowed with the norm associated to the canonical basis. It is the norm associated to the integral model $\m{M}_2(\Z_2) \subset \m{M}_2(\Q_2)$ and the $2$-adic filtration $\m{M}_2(\Z_p) \supsetneq 2 \m{M}_2(\Z_2) \supsetneq \cdots \supsetneq \bigcap_{i \in \N} 2^i \m{M}_2(\Z_2) = O$. The reduction of $\m{M}_2(\Q_2)$ with respect to the norm is $\m{M}_2(\F_2)$, and it is surely a simple $\F_2$-algebra. On the other hand, $\m{M}_2(\Q_2)$ admits another equivalent norm. Consider the norm associated to the $2$-adic filtered integral model
\begin{eqnarray*}
  \left(
    \begin{array}{cc}
      \Z_2 & \Z_2 \\
      2\Z_2 & \Z_2
    \end{array}
  \right)
  \supsetneq 2
  \left(
    \begin{array}{cc}
      \Z_2 & \Z_2 \\
      2\Z_2 & \Z_2
    \end{array}
  \right)
  \supsetneq \cdots \supsetneq \bigcap_{i \in \N} 2^i
  \left(
    \begin{array}{cc}
      \Z_2 & \Z_2 \\
      2\Z_2 & \Z_2
    \end{array}
  \right)
  = O.
\end{eqnarray*}
The reduction $R$ of $\m{M}_2(\Q_2)$ with respect to the norm is naturally isomorphic to the $\F_2$-subalgebra of $\m{M}_2(\F_2^2) \cong_{\F_2-\m{Alg}}\m{M}_2(\F_2[X]/\F_2[X](X^2+X))$ spanned by
\begin{eqnarray*}
  e_{11} \coloneqq
  \left(
    \begin{array}{cc}
      1 & 0 \\
      0 & 0 
    \end{array}
  \right),
  e_{12} \coloneqq
  \left(
    \begin{array}{cc}
      0 & X \\
      0 & 0 
    \end{array}
  \right),
  e_{21} \coloneqq
  \left(
    \begin{array}{cc}
      0 & 0 \\
      X+1 & 0 
    \end{array}
  \right),
  e_{22} \coloneqq
  \left(
    \begin{array}{cc}
      0 & 0 \\
      0 & 1 
    \end{array}
  \right).
\end{eqnarray*}
The $\F_2$-vector subspace
\begin{eqnarray*}
  V \coloneqq \F_2
  \left(
    \begin{array}{c}
      X \\
      0
    \end{array}
  \right)
  \oplus \F_2
  \left(
    \begin{array}{c}
      0 \\
      1
    \end{array}
  \right) \subset \left( \F_2[X]/\F_2(X^2+X) \right)^2
\end{eqnarray*}
is stable under the action of $R$, and the matrix representations $T_{11}, T_{12}, T_{21}, T_{22} \in \m{M}_2(\F_2)$ of $e_{11}, e_{12}, e_{21}, e_{22}$ on $V$ with respect to the $\F_2$-basis above are
\begin{eqnarray*}
  T_{11} \coloneqq
  \left(
    \begin{array}{cc}
      1 & 0 \\
      0 & 0 
    \end{array}
  \right),
  T_{12} \coloneqq
  \left(
    \begin{array}{cc}
      0 & 1 \\
      0 & 0 
    \end{array}
  \right),
  T_{21} \coloneqq
  \left(
    \begin{array}{cc}
      0 & 0 \\
      0 & 0 
    \end{array}
  \right),
  T_{22} \coloneqq
  \left(
    \begin{array}{cc}
      0 & 0 \\
      0 & 1 
    \end{array}
  \right).
\end{eqnarray*}
Therefore $V$ is not completely reducible as a left $R$-module. Thus $R$ is not a semisimple $\F_2$-algebra.

\begin{dfn}
Let $F$ be a field. An $F$-algebra $A$ is said to be pro-semisimple if there is a faithful semisimple left $A$-module $M$ such that every simple submodule of $M$ is of finite dimension and the image of $A$ in $\m{End}_F(M)$ is strongly closed with respect to the trivial valuation of $F$ and the trivial norm of $M$.
\end{dfn}

The strong closedness in the definition is equivalent with the weak closedness because $M$ is a direct sum of finite dimensional simple left $A$-modules. It follows from Jacobson--Bourbaki density theorem (\cite{Cri04} D 2.2) that $A$ is isomorphic to the direct product of simple $F$-algebras given as the double commutant $\m{End}_{\m{End}_F(\mu)}(\mu)$ for a representative of each isomorphism class of simple left $A$-submodules of $M$. In particular, presenting the identity as the sequence of the identity of the simple $F$-algebras appearing in the decomposition, we have a system of orthonormal primitive central idempotents of $A$.

\begin{rmk}
If $A$ is of finite dimension, then the pro-semisimplicity is equivalent to the semisimplicity.
\end{rmk}

\begin{rmk}
A pro-semisimple $F$-algebra $A$ is a semisimple $F$-algebra if and only if $A$ is of finite dimension. Indeed, if $A$ is of infinite dimension, its centre $A'$ is a direct product of infinitely many fields of finite dimension over $F$. The spectrum of $A'$ is Stone--$\check{\m{C}}$ech compactification of the discrete space given as the disjoint union of the spectra of the fields, and it possesses a point corresponding to a non-principal ultrafilter of the discrete set. Such a point corresponds to a non-projective maximal ideal of $A'$, and its commutant is a two-sided ideal of $A$ which is not generated by a central idempotent. Thus $A$ is not a semisimple $F$-algebra.
\end{rmk}

We verify the relation between a certain topological semisimplicity of a Banach $k$-algebra $A$ and the pro-semisimplicity of its reduction $\overline{\A}$.

\begin{thm}
\label{topological semisimplicity}
Let $\A$ be a Banach $k$-algebra with $\n{\A} = \v{k}$. If $\overline{\A}$ is a pro-semisimple $\overline{k}$-algebra, then $\A$ admits a canonical dense two-sided ideal $\A_{\circ}$ which is a direct sum of the underlying left $(\A \times \A^{\m{op}})$-modules of simple $\A$-algebras and whose simple components are of finite dimension. Moreover, the decomposition of $\A_{\circ}$ into simple components is derived from a unique decomposition of $\A(1) \cap \A_{\circ}$ into indecomposable projective two-sided ideals.
\end{thm}

In particular, Theorem states that the reduction respects the semisimplicity in the finite dimensional case.

\begin{proof}
Since $\overline{\A}$ is pro-semisimple, it admits a subset $\overline{E}$ of central idempotents such that $1 = \sum_{\overline{e} \in \overline{E}} \overline{e}$ , $\overline{e} \overline{e}{}' = 0$ for any $(\overline{e},\overline{e}{}') \in (\overline{E}{}')^2$ with $\overline{e} \neq \overline{e}{}'$, and $\overline{\A} \overline{e}$ is a simple ring of finite dimension for any $\overline{e} \in \overline{E}$. Here the sum $\sum_{\overline{e} \in \overline{E}} \overline{e}$ means the limit $\lim_{\overline{S} \in \Fil(\overline{E})} \sum_{\overline{e} \in \overline{S}} \overline{e}$ in the strong topology of $\m{End}_{\overline{k}}({}_{\overline{\A}}\overline{\A})$ along the directed set $\Fil(\overline{E}) \subset 2^{\overline{E}}$ of finite subsets. Now $\overline{E}$ is the set of primitive central idempotents of $\overline{\A}$, and hence is independent of the choice of a faithful semisimple left $\A$-module $M$ in the definition of the pro-semisimplicity. By Corollary \ref{lift 2}, $\A$ admits a subset $E$ of central idempotents such that $1 = \sum_{e \in E} e$, $e \in \A(1)$, and the correspondence $e \mapsto e + \A(1-) \in \overline{\A}$ gives a bijective map $E \to \overline{E}$. Here the sum $\sum_{e \in E} e$ means the strong limit again but not the limit in the norm topology. Moreover, since $(e + \A(1-))(e' + \A(1-)) = 0$ and $e e' = e' e$, we have $e e' = 0$ for any $(e,e') \in E^2$ with $e \neq e'$. Indeed, $e e'$ is an idempotent with $e e' \in \A(1-)$. Every element of $\A(1-)$ is topologically nilpotent, and an idempotent is topologically nilpotent if and only if it is zero.

Thus we have obtained a semisimple two-sided ideal $\A_{\circ} \coloneqq \bigoplus_{e \in E} \A e \subset \A$, and it is dense because the directed system $(\sum_{e \in S} e)_{S \in \Fil(E)}$ of central idempotents along the directed set $\Fil(E) \subset 2^E$ of finite subsets forms an approximate unit. The decomposition of $\A_{\circ}$ is the orthonormal direct sum of normed $k$-vector spaces because it is derived from the system of orthonormal idempotents with norm $1$. This gives a decomposition $\A(1) \cap \A_{\circ} = \bigoplus_{e \in E} \A(1) e$. The reduction of $\A e$ is the simple $\overline{\A}$-algebra $\overline{\A}(e + \A(1-))$ for any $e \in E$. This completes the proof by Theorem \ref{simplicity}.
\end{proof}

We remark that the norm of $\A$ is restored from the canonical dense integral model $\A(1) \cap \A_{\circ}$. Indeed, $\A(1)$ coincides with the $\varpi$-adic completion of $\A(1) \cap \A_{\circ}$ endowed with the $\varpi$-adic norm because $\A(1-) \cap \A_{\circ} = \varpi (\A(1) \cap \A_{\circ})$, where $\varpi \in k^{\circ \circ}$ is a uniformiser.

\begin{crl}
\label{semisimplicity}
Let $\A$ be a Banach $k$-algebra with $\n{\A} = \v{k}$. If $\overline{\A}$ is a finite dimension semisimple $\overline{k}$-algebra, then the underlying $k$-algebra of $\A$ is a finite dimensional semisimple $k$-algebra. Moreover, the decomposition of $\A$ into simple $\A$-algebras is derived from a unique decomposition of $\A(1)$ into indecomposable projective two-sided ideals.
\end{crl}

A counterexample of the converse of Corollary \ref{semisimplicity} is given by a unitary representation of a $p$-group. For example, consider the group algebra $\Q_2[\Z/2\Z]$. It is semisimple because $\m{ch}(\Q_2) = 0$. It admits the norm associated to the integral model $\Z_2[\Z/2\Z] \subset \Q_2[\Z/2\Z]$ and the $2$-adic filtration $\Z_2[\Z/2\Z] \supsetneq 2 \Z_2[\Z/2\Z] \supsetneq \cdots \supsetneq \bigcap_{i \in \N} 2^i \Z_2[\Z/2\Z] = O$, and its reduction with respect to the norm is the group algebra $\F_2[\Z/2\Z] \cong_{\F_2-\m{Alg}} \F_2[X]/\F_2[X](X^2+1) \cong_{\F_2-\m{Alg}} \F_2[Y]/\F_2[Y]Y^2$. It is a local commutative ring which is not a field, and hence is not a simple ring. It is remarkable that $\Q_2[\Z/2\Z]$ admits another equivalent norm which gives a semisimple reduction. Put $\sigma \coloneqq [0 + 2\Z] + [1 + 2\Z] \in \Z_2[\Z/2\Z]$. Consider the norm associated to the integral model $\Z_2[\Z/2\Z][2^{-1}\sigma] \subset \Q_2[\Z/2\Z]$ and the $2$-adic filtration $\Z_2[\Z/2\Z][2^{-1}\sigma] \supsetneq 2 \Z_2[\Z/2\Z][2^{-1}\sigma] \supsetneq \cdots \supsetneq \bigcap_{i \in \N} 2^i \Z_2[\Z/2\Z][2^{-1}\sigma] = O$. This is the operator norm with respect to the regular representation identified with the orthogonal direct sum of the two characters $\Z/2\Z \to \Q_2^{\times} \colon [1 + 2\Z] \mapsto \pm 1$. The reduction of $\Q_2[\Z/2\Z]$ with respect to the norm is $\F_2^2$, and this is a semisimple $\F_2$-algebra.

\section{Connection to Representation Theory}
\label{Connection to Representation Theory}

We continue to assume that the base field $k$ is a local field. We apply the results in \S \ref{Reductively Semisimple Banach Algebras} to the operator algebra $\A$ associated to a unitary representation of a topological monoid $\M$. As we referred in \S \ref{Introduction}, the integral model $\A(1)$ of $\A$ possesses enough operators unlike the image of the integral model $k^{\circ}[\M] \subset k[\M]$ so that the reduction respects the semisimplicity of the natural left module. There is a problem that it is a little difficult to determine the structure of the operator algebra $\A$ and hence the semisimplicity of the reductive operator algebra $\overline{\A}$ in a direct way. We establish a way to calculate $\overline{\A}$ without determining $\A$ by a repetitive reduction method. This algorithm might contain infinitely many steps in general, but when we deal with a finite dimensional representation, then the algorithm stops in finite steps.

\subsection{Semisimplicity of a Unitary Representation}
\label{Semisimplicity of a Unitary Representation}

We apply the result of \S \ref{Reductively Semisimple Banach Algebras} to an operator algebra associated to a unitary representation of a topological monoid $\M$. This gives a criterion of the semisimplicity of the representation. The reduction of a representation itself does not preserve the semisimplicity. The unit ball of the operator algebra is larger than the image of the integral model $k^{\circ}[\M]$ of $k[\M]$, and its reduction possesses enough operators for the semisimplicity of the natural left module to be preserved, while the image of $\overline{k}[\M]$ does not.

\begin{thm}
\label{representation}
Let $\M$ be a topological monoid, and $(V,\rho)$ a finite dimensional strictly Cartesian unitary representation of $\M$ over $k$. Denote by $\A$ the closure of the image of $k[\M]$ by the $k$-linear extension of $\rho$ in the Banach $k$-algebra $\m{B}_k(V)$ of continuous $k$-linear endomorphisms endowed with the operator norm. If $\overline{\A}$ is a semisimple $\overline{k}$-algebra, then $(V,\rho)$ is a semisimple representation of $\M$.
\end{thm}

\begin{proof}
Since $V$ is of finite dimension, so is $\overline{\A}$. Therefore $\A$ is a finite dimensional semisimple $k$-algebra by Corollary \ref{semisimplicity}, and $\A$ admits central idempotents $e_1, \ldots, e_m \in \A(1)$ such that $1 = e_1 + \cdots + e_m$, $e_i e_j = 0$ for any $(i,j) \in \ens{1, \ldots, m}$ with $i \neq j$, and $\A e_i$ is an Artinian simple $k$-algebra with precisely one isomorphism class of simple left modules for any $i \in \ens{1, \ldots, m}$. Then $V$ decomposes into simple left $\A$-submodules. Let $W \subset V$ be a simple left $\A$-submodule. Since $\A$ contains the image of $k[\M]$, $W$ is a left $k[\M]$-submodule. We denote by $\rho_W \colon \M \to \m{End}_k(W)$ the restriction of $\rho$ on $W$, and verify that $(W,\rho_W)$ is a irreducible representation of $\M$ over $k$. Let $W' \subset W$ be a subrepresentation of $\M$. Then $W'$ is a left $k[\M]$-submodule of $W$, and $aw$ is contained in the closure of $W'$ for any $(a,w) \in \A \times W'$. On the other hand, $W' \subset W$ is closed because $W$ is of finite dimension, and hence $W'$ is a left $\A$-submodule of $W$. Since $W$ is a simple left $\A$-module, we obtain $W' = O$ or $W' = W$. Thus $W$ is a irreducible representation of $\M$. We conclude that $V$ decomposes into irreducible subrepresentations of $\M$.
\end{proof}

Theorem \ref{representation} is a partial generalisation of \cite{Mih} Theorem 5.7. A representation of a single operator corresponds to a representation of the discrete Abelian monoid $\N$, and \cite{Mih} Theorem 5.7 gives a criterion of the semisimplicity of the corresponding representation by the repetition of finitely many reductions. The reason why we considered a reduction only once in Theorem \ref{representation} is because we deal with the case when $\overline{\A}$ is known. The repetition of reductions corresponds to the repetitive calculation necessary to determine the reductive operator algebra $\overline{\A}$. The following explains the correspondence.

Let $\M$ be a topological monoid, and $(V,\rho)$ a strictly Cartesian unitary representation of $\M$ over $k$. Since the valuation of $k$ is discrete and $\n{V} \subset \v{k}$, the operator norm of $\m{B}_k(V)$ coincides with the supremum norm of coefficients of the matrix presentation with respect to an orthonormal Schauder basis, and $\overline{\m{B}_k(V)}$ is naturally isomorphic to $\m{End}_{\overline{k}}(\overline{V})$. Denote by $\Pi \colon \m{B}_k(V)(1) \twoheadrightarrow \m{End}_{\overline{k}}(\overline{V})$ the canonical projection. Take a uniformiser $\varpi \in k^{\circ \circ}$. Set $A_0 \coloneqq k^{\circ}[\M] \subset k[\M]$. Remark that the image of $A_0$ by the $k$-linear extension $k[\rho]$ of $\rho$ is contained in $\m{B}_k(V)(1)$. We define $A_i \subset k[\rho]^{-1}(\m{B}_k(V)(1))$ in an inductive way on $i \in \N$. Suppose that $A_i \subset k[\rho]^{-1}(\m{B}_k(V)(1))$ is defined for an $i \in \N$. Then we set $A_{i+1} \coloneqq A_i + \varpi^{-1} \ker(\Pi \circ k[\rho]|_{A_i}) \subset k[\rho]^{-1}(\m{B}_k(V)(1))$. For each $i \in \N$, we put $\alpha_i \coloneqq \im(\Pi \circ k[\rho]|_{A_i}) \subset \m{End}_{\overline{k}}(\overline{V})$.

\begin{prp}
The reduction $\overline{\A}$ of the closure $\A$ of $\im(k[\rho]) \subset \m{B}_k(V)$ coincides with $\bigcup_{i \in \N} \alpha_i$.
\end{prp}

\begin{proof}
Let $f \in k[\rho]^{-1}(\m{B}_k(V)(1))$. If $f \in k^{\circ}[\M] = A_0$, then $f \in \bigcup_{i \in \N} A_0$. Otherwise, take an $h \in \N$ such that $\varpi^{h+1} f \in k^{\circ}[\M] \backslash \varpi k^{\circ}[\M]$, where $\varpi \in k^{\circ \circ}$ is a uniformiser. Then $\varpi^{h+1-i}f \in \ker(\Pi \circ k[\rho]|_{A_i})$ for any $i \in \ens{0, \ldots, h}$ and $f \in A_{h+1} \subset \bigcup_{i \in \N} A_0$. Thus $k[\rho]^{-1}(\m{B}_k(V)(1)) = \bigcup_{i \in \N} A_0$. We conclude
\begin{eqnarray*}
  & & \overline{\A} = \im \left( \Pi \circ k[\rho]|_{k[\rho]^{-1}(\m{B}_k(V)(1))} \right) = \im \left( \Pi \circ k[\rho]|_{\bigcup_{i \in \N} A_i} \right) = \bigcup_{i \in \N} \im \left( \Pi \circ k[\rho]|_{A_i} \right) \\
  & = & \bigcup_{i \in \N} \alpha_i.
\end{eqnarray*}
\end{proof}

\begin{crl}
\label{repetition}
Suppose that there is an $n \in \N$ such that the increasing sequence
\begin{eqnarray*}
  \sum_{m \in \M} \overline{k} \overline{\rho}(m) = \alpha_0 \subset \alpha_1 \subset \alpha_2 \subset \cdots
\end{eqnarray*}
of $\overline{k}$-vector subspaces of $\m{End}_{\overline{k}}(\overline{V})$ satisfies $\alpha_n = \alpha_{n+1} = \cdots$. Then the reduction $\overline{\A}$ of the closure $\A$ of $\im(k[\rho]) \subset \m{B}_k(V)$ coincides with $\alpha_n$.
\end{crl}

In particular, the stability condition holds for a finite dimensional representation. Thus $\overline{\A}$ can be calculated in finite steps with the reductions. These correspond to the repetition of the reductions in \cite{Mih} Theorem 5.7 and Theorem 5.20. We will compute $\overline{\A}$ for several basic examples in \S \ref{Examples}.

\subsection{Examples}
\label{Examples}

We give several basic examples of the calculation of the reduction $\overline{\A}$ of the operator algebra $\A$ associated to a unitary representation in the way with the repetitive reduction method we considered in \S \ref{Semisimplicity of a Unitary Representation}.

\begin{exm}
Consider the strictly Cartesian unitary representation
\begin{eqnarray*}
  \rho \colon \Z_2 & \to & \m{End}_{\Q_2}(\Q_2^2) \cong_{\Q_2-\m{Alg}} \m{M}_2(\Q_2) \\
  a & \mapsto &
  \left(
    \begin{array}{cc}
      1 & a \\
      0 & 1
    \end{array}
  \right)
\end{eqnarray*}
of the topological Abelian monoid $\Z_2$ over $\Q_2$, where $\Q_2^2$ is endowed with the norm associated to the canonical basis. Then $(\Q_2^2,\rho)$ is not semisimple. The closed $\Q_2$-algebra $\A \coloneqq \Q_2[\rho](\Q_2[\Z_2])$ can be easily computed as
\begin{eqnarray*}
  \A = \Q_2
  \left(
    \begin{array}{cc}
      1 & 0 \\
      0 & 1
    \end{array}
  \right)
  \oplus \Q_2
  \left(
    \begin{array}{cc}
      0 & 1 \\
      0 & 0
    \end{array}
  \right),
\end{eqnarray*}
and this is isomorphic to $\Q_2[X]/\Q_2[X]X^2$, which is not a semisimple ring. Its reduction is isomorphic to $\F_2[X]/\F_2[X]X^2$, and it is not a semisimple ring, either.
\end{exm}

\begin{exm}
Let $\M$ be the free product $\Z_2 * \Z_2$ of the copies of the underlying group of $\Z_2$. We denote by $\iota_1$ (resp.\ $\iota_2$) the embedding $\Z_2 \hookrightarrow \M$ as the first (resp.\ second) component. Consider the strictly Cartesian unitary representation
\begin{eqnarray*}
  \rho \colon \M & \to & \m{End}_{\Q_2}(\Q_2^2) \cong_{\Q_2-\m{Alg}} \m{M}_2(\Q_2) \\
  \iota_1(a) & \mapsto &
  \left(
    \begin{array}{cc}
      1 & a \\
      0 & 1
    \end{array}
  \right) \\
  \iota_2(a) & \mapsto &
  \left(
    \begin{array}{cc}
      1 & 0 \\
      a & 1
    \end{array}
  \right)
\end{eqnarray*}
of the discrete monoid $\M$ over $\Q_2$, where $\Q_2^2$ is endowed with the norm associated to the canonical basis. Then $(\Q_2^2,\rho)$ is irreducible. By the simplicity of $\Q_2^2$ as a left $\Q_2[\M]$-module, the closed $\Q_2$-algebra $\A \coloneqq \Q_2[\rho](\Q_2[\M])$ is the full matrix algebra $\m{M}_2(\Q_2)$ by Jacobson--Bourbaki density theorem (\cite{Cri04} D 2.2). This fact guarantees that $\overline{\A}$ is the full matrix algebra $\m{M}_2(\F_2)$, which is a simple $\F_2$-algebra. Indeed, the reduction $\overline{\rho}$ of $\rho$ is given as
\begin{eqnarray*}
  \overline{\rho} \colon \Z_2 * \Z_2 & \to & \m{End}_{\F_2}(\F_2^2) \cong_{\F_2-\m{Alg}} \m{M}_2(\F_2) \\
  \iota_1(a) & \mapsto &
  \left(
    \begin{array}{cc}
      1 & a + 2\Z_2 \\
      0 & 1
    \end{array}
  \right) \\
  \iota_2(a) & \mapsto &
  \left(
    \begin{array}{cc}
      1 & 0 \\
      a + 2\Z_2 & 1
    \end{array}
  \right),
\end{eqnarray*}
and hence
\begin{eqnarray*}
  & & \alpha_0 \supset \F_2 + \sum_{i = 1,2} \F_2 (\Pi \circ \F_2[\overline{\rho}])([\iota_i(1)] - 1) + \F_2 (\Pi \circ \F_2[\overline{\rho}])(([\iota_1(1)] - 1)([\iota_2(1)] - 1)) \\
  & = &
  \F_2
  \left(
    \begin{array}{cc}
      1 & 0 \\
      0 & 1
    \end{array}
  \right)
  \oplus \F_2
  \left(
    \begin{array}{cc}
      0 & 1 \\
      0 & 0
    \end{array}
  \right)
  \oplus \F_2
  \left(
    \begin{array}{cc}
      0 & 0 \\
      1 & 0
    \end{array}
  \right)
  \oplus \F_2
  \left(
    \begin{array}{cc}
      1 & 0 \\
      0 & 0
    \end{array}
  \right) \\
  & = & \m{M}_2(\F_2).
\end{eqnarray*}
Thus $\overline{\A} = \m{M}_2(\F_2)$.
\end{exm}

\begin{exm}
Let $\M, \iota_1, \iota_2$ be as above. Consider the strictly Cartesian unitary representation
\begin{eqnarray*}
  \rho \colon \M & \to & \m{End}_{\Q_2}(\Q_2^2) \cong_{\Q_2-\m{Alg}} \m{M}_2(\Q_2) \\
  \iota_1(a) & \mapsto &
  \left(
    \begin{array}{cc}
      1 & a \\
      0 & 1
    \end{array}
  \right) \\
  \iota_2(a) & \mapsto &
  \left(
    \begin{array}{cc}
      1 & 0 \\
      2a & 1
    \end{array}
  \right)
\end{eqnarray*}
of the discrete monoid $\M$ over $\Q_2$, where $\Q_2^2$ is endowed with the norm associated to the canonical basis. Then $(\Q_2^2,\rho)$ is irreducible. By the simplicity of $\Q_2^2$ as a left $\Q_2[\M]$-module, the closed $\Q_2$-algebra $\A \coloneqq \Q_2[\rho](\Q_2[\M])$ is the full matrix algebra $\m{M}_2(\Q_2)$ by Jacobson--Bourbaki density theorem (\cite{Cri04} D 2.2). This fact guarantees that $\overline{\A}$ is the full matrix algebra $\m{M}_2(\F_2)$, which is a simple $\F_2$-algebra. We calculate $\overline{\A}$ without use of the irreducibility in the way in Corollary \ref{repetition}. The reduction $\overline{\rho}$ of $\rho$ is given as
\begin{eqnarray*}
  \overline{\rho} \colon \Z_2 * \Z_2 & \to & \m{End}_{\F_2}(\F_2^2) \cong_{\F_2-\m{Alg}} \m{M}_2(\F_2) \\
  \iota_1(a) & \mapsto &
  \left(
    \begin{array}{cc}
      1 & a + 2\Z_2 \\
      0 & 1
    \end{array}
  \right) \\
  \iota_2(a) & \mapsto &
  \left(
    \begin{array}{cc}
      1 & 0 \\
      0 & 1
    \end{array}
  \right),
\end{eqnarray*}
and hence
\begin{eqnarray*}
  \alpha_0 = \F_2
  \left(
    \begin{array}{cc}
      1 & 0 \\
      0 & 1
    \end{array}
  \right)
  \oplus \F_2
  \left(
    \begin{array}{cc}
      0 & 1 \\
      0 & 0
    \end{array}
  \right).
\end{eqnarray*}
Moreover, $\ker(\Pi \circ \Q_2[\rho]|_{\A_0})$ contains $[\iota_2(1)] - 1$ and $([\iota_2(1)] - 1)([\iota_1(1)] - 1)$, and hence $\Z_2[\Z_2 * \Z_2] + \Z_2(2^{-1}([\iota_2(1)] - 1)) + \Z_2 2^{-1}([\iota_2(1)] - 1)([\iota_1(1)] - 1) \subset A_1$. It implies that $\alpha_1$ contains the $\F_2$-vector subspace
\begin{eqnarray*}
  & & \alpha_0 + \F_2 (\Pi \circ \Q_2[\rho]|_{A_0}) \left( \frac{[\iota_2(1)] - 1}{2} \right) + \F_2 (\Pi \circ \Q_2[\rho]|_{A_0}) \left( \frac{([\iota_2(1)] - 1)([\iota_1(1)] - 1)}{2} \right) \\
  & = & \alpha_0 + \F_2
  \left(
    \begin{array}{cc}
      0 & 0 \\
      1 & 0
    \end{array}
  \right)
  +
  \F_2
  \left(
    \begin{array}{cc}
      1 & 0 \\
      0 & 0
    \end{array}
  \right) \\
  & = & \F_2
  \left(
    \begin{array}{cc}
      1 & 0 \\
      0 & 0
    \end{array}
  \right)
  \oplus \F_2
  \left(
    \begin{array}{cc}
      0 & 0 \\
      0 & 1
    \end{array}
  \right)
  \oplus \F_2
  \left(
    \begin{array}{cc}
      0 & 1 \\
      0 & 0
    \end{array}
  \right) \oplus \F_2
  \left(
    \begin{array}{cc}
      0 & 0 \\
      1 & 0
    \end{array}
  \right) = \m{M}_2(\F_2).
\end{eqnarray*}
Thus we have succeeded in computing $\overline{\A} = \m{M}_2(\F_2)$.
\end{exm}

\section{$p$-adic Unitary Dual}
\label{$p$-adic Unitary Dual}

We continue to assume that the base field $k$ is a local field. We observe the relation between the central idempotents arising in the repetitive reduction method in the calculation of the reductive operator algebra $\A$ associated to an infinite dimensional semisimple multiplicity free unitary representation of a profinite group $G$ and the topology of the $p$-adic unitary dual $\check{G}$ of $G$. The case is much simpler when $G$ is an Abelian profinite group. This observation connects the repetitive reduction method to Amice's theory of Fourier transform.

\subsection{Refined Fell Topology}
\label{Refined Fell Topology}

In this subsection, let $G$ be a profinite group. We introduce the notion of the $p$-adic dual $\check{G}$ of $G$. We endow it with a certain topology finer than the ordinary topology. The definition of the $p$-adic unitary dual is easily extended to that of a locally profinite group, but we see only a profinite group in this paper.

\begin{dfn}
We denote by $\check{G}_k$ the set of isomorphism classes of finite dimensional strictly Cartesian irreducible unitary representations of $G$ over $k$ . We endow $\check{G}_k$ with the topology generated by subsets of the following form:
\begin{eqnarray*}
  U_{(V,\rho),r,S,S'} \coloneqq \Set{I \in \check{G}}
  {
  \begin{array}{l}
    {}^{\exists}(W,\pi) \in I, {}^{\exists} \iota_1 \colon S \to W, {}^{\exists} \iota_2 \colon S' \to \m{Hom}_k(W,k), s.t.\\
    \v{s'(k[\rho](a)(s)) - \iota_2(s')(k[\pi](a)(\iota_1(s)))} \leq r, \\
    {}^{\forall} (a,s,s') \in k^{\circ}[G] \times S \times S'
  \end{array}
  }
\end{eqnarray*}
where $(V,\rho)$ is a finite dimensional strictly Cartesian irreducible unitary representation of $G$ over $k$, $r \in ( 0, 1]$, $S \subset V$ is a finite subset, and $S' \subset \m{Hom}_k(V,k)$ is a finite subset.
\end{dfn}

The class $\check{G}_k$ is not a proper class because every finite dimensional unitary representation of $G$ over $k$ is presented as a continuous group homomorphism $G \to \m{GL}_n(k^{\circ})$. We remark that $\check{G}_k$ has enough points because admissible representations of $G$ over $k$ separates points of $G$. In the definition of $U_{(V,\rho),r,S,S'}$, one may naturally replace $k^{\circ}[G]$ by the Iwasawa algebra $k^{\circ}[[G]]$, which is a compact Hausdorff linear topological $k^{\circ}$-algebra. This topology is finer than Fell topology. Such refinement is not useful for the unitary dual over $\C$ because $\C[G]$ is not totally bounded for any Hausdorff locally convex topology.

For every $I \in \check{G}$, take a representative $(V_I,\rho_I) \in I$. We denote by $V$ the completion of $\bigoplus_{I \in \check{G}_k} V_I$ regarded as the orthogonal direct sum of normed $k$-vector spaces. The topology of $\check{G}_k$ is $\m{T}_1$ by Jacobson-Bourbaki density theorem (\cite{Cri04} D 2.2). In other words, the $k$-algebra homomorphism $\prod_{I \in \check{G}_k} \rho_I \colon k[G] \to k \otimes_{k^{\circ}} \prod_{I \in \check{G}_k} \m{B}_k(V_I)(1) \subset \m{B}_k(V)$ is injective. The density theorem guarantees that the image is strongly dense. We denote by $\A$ (resp.\ $A_0$) the closure of the image of $k[G]$ (resp.\ $k^{\circ}[G]$) in the norm topology. For every integral model $O \subset k[G]$, every orthonormal system of central idempotents of the strong closure of the image of gives a partition of $\check{G}_k$ into clopen subsets. In particular, the lift $E_0 \subset A_0 \subset \prod_{I \in \check{G}_k} \m{B}_k(V_I)(1)$ of the set of primitive central idempotent of $\alpha_0 \coloneqq A_0/(A_0 \cap \A(1)) \subset \prod_{I \in \check{G}_k} \m{End}_{\overline{k}} \overline{V}$ given by Corollary \ref{lift 2} as in the proof of Theorem \ref{topological semisimplicity} yields a canonical partition of $\check{G}_k$. This is a generalisation of the block decomposition of the unitary dual of a finite group. Moreover, defining $A_i$ in a inductive way on $i \in \N$ similar with that in \S \ref{Semisimplicity of a Unitary Representation}, we obtain a refinement sequence of partitions of $\check{G}_k$. This repetition of infinitely many refinements corresponds to the repetition of infinitely many reductions in \cite{Mih} as is observed in Proposition \ref{repetition} in the finite dimensional case.

We will observe the most basic example of the structure of this system in \S \ref{Relation to Amice's Theory}. A system of partitions by clopen subsets works well for a non-Archimedean uniform space. Here a uniform space is said to be {\it non-Archimedean} if it admits a fundamental system of entourages consisting of equivalence relations. A profinite space has a canonical Hausdorff non-Archimedean uniform structure, and hence a restriction of the system on a profinite subset helps us to understand it well.

\subsection{Relation to Amice's Theory}
\label{Relation to Amice's Theory}

Let $G$ be the profinite group $\Z_p$. We fix an algebraic closure $C$ of $\Q_p$ and denote by $\C_p$ the completion of $C$ with respect to the norm associated to the unique extension of the valuation of $\Q_p$. There is a one-to-one correspondence between the open unit disc $1 + \C_p^{\circ \circ}$ centred at $1 \in \C_p$ and the set of continuous characters $\Z_p \to \C_p^{\times}$ sending an $a \in 1 + \C_p^{\circ \circ}$ to the character $\chi_a \colon \Z_p \to \C_p^{\times}$ with $\chi_a(1) = a$. The latter set coincides with the set of the isomorphism classes of finite dimensional strictly Cartesian irreducible representation of $\Z_p$ over $C$ by Schur's lemma. Let $k/\Q_p$ be a local field contained in $\C_p$. Then every continuous character of $\Z_p$ on $k$ corresponds to the open unit disc $1 + k^{\circ \circ} \subset 1 + \C_p^{\circ \circ}$. Other finite dimensional strictly Cartesian irreducible unitary representations of $\Z_p$ on $k$ are not absolutely irreducible and correspond to conjugacy classes of $1 + C^{\circ \circ} \subset 1 + \C_p^{\circ \circ}$ with respect to the natural action of the absolute Galois group $\m{Gal}(C/k)$. Thus we obtain a bijective map $\check{G}_k \to (1 + C)/\m{Gal}(C/k)$. Beware that every finite extension $K/k$ is not necessarily strictly Cartesian with respect to the norm induced by the unique extension of the valuation of $k$, and hence one needs to consider an equivalent norm as a $k$-vector space. However, as we remarked at the end of \S \ref{Unitary Representation of a Topological Monoid}, the norm of the underlying Banach $k$-vector space is not an invariant of an isomorphism class of representations unlike the equivalence class of norms. Therefore the norm of $K$ associated to the valuation works well when we calculate the $k$-rational descent $\Z_p \to \m{Aut}_k(K)$ of a $K$-rational character $\Z_p \to K^{\times}$ concretely. We also remark that the fundamental system $(U_{(V,\rho),r,S,S'})_{(V,\rho),r,S,S'}$ of the topology of $\check{G}_k$ does not reflect the norm of $V$ in the parameter.

By the argument above, calculation of the topology of $\check{G}_k$ using the valuation of $\overline{Q}_p$ guarantees that the set-theoretical identification $\check{G}_k \to (1 + C)/\m{Gal}(C/k)$ is a homeomorphism. In particular, the subset $\check{G}(k) \subset \check{G}_k$ of isomorphism classes of continuous characters is homeomorphic to the open unit disc $1 + k^{\circ \circ}$. We compute the restriction of the system of partitions on $\check{G}(k)$ given at the end of \S \ref{Refined Fell Topology}. We follow the notation in \S \ref{Refined Fell Topology}. By Amice's theory of Fourier transform of $\Z_p$, $\A \subset \prod_{a \in 1 + C^{\circ \circ}} \m{B}_k(k(a))(1)$ coincides with a strongly dense $k$-algebra of the Banach $k$-algebra $k[[\Z_p]] \cong_{\m{Ban} k-\m{Alg}} k[[T - 1]]$ of formal power series regarded as a closed $k$-subalgebra $\m{C}_{\m{bd}}(1 + C,\C_p)$ of bounded continuous $\C_p$-valued functions on $1 + C$. The restriction of the supremum norm on $k[[T - 1]]$ coincides with the Gauss norm. The integral model $A_0$ coincides with a strongly dense $k^{\circ}$-subalgebra of $k^{\circ}[[\Z_p]] \cong_{\m{Top} k^{\circ}-\m{Alg}} k^{\circ}[[T - 1]]$. Since $k^{\circ}[[T - 1]]$ is an integral domain, the partition of $\check{G}(k) = 1 + k^{\circ \circ}$ corresponding to $A_0$ is trivial. Take a uniformiser $\varpi \in k^{\circ \circ}$. For each $i \in \N$, the integral model $A_i$ contains $\varpi^{-i}(T - 1)^i \in k[[T - 1]]$, and the corresponding partition is finer than or equal to the partition given by the quotient modulo $\varpi^{i+1} k^{\circ}$, i.e.\ the canonical projection
\begin{eqnarray*}
  1 + k^{\circ \circ} = 1 + \varpi k^{\circ} = \bigsqcup_{\sigma \in \varpi k^{\circ}/\varpi^{i + 1} k^{\circ}} 1 + \sigma \twoheadrightarrow 1 + \varpi k^{\circ}/\varpi^{i + 1}k^{\circ} \subset k^{\circ}/\varpi^{i + 1}k^{\circ}.
\end{eqnarray*}
On the other hand, starting from $k^{\circ}[[T - 1]]$, we define an increasing  filtration $B_0 \subset B_1 \subset B_2 \subset \cdots$ dominating $A_0 \subset A_1 \subset A_2 \subset \cdots$ as
\begin{eqnarray*}
  & & k^{\circ}[[T - 1]] \subset k^{\circ}[[T - 1]] + k^{\circ}\frac{T-1}{\varpi} \subset k^{\circ}[[T - 1]] + k^{\circ}\frac{T-1}{\varpi} + k^{\circ}\frac{(T - 1)^2}{\varpi^2} \\
  & \subset & \cdots \subset k^{\circ} \left[ \middle[ \frac{T - 1}{\varpi} \middle] \right].
\end{eqnarray*}
The system of partitions associated to $B_0 \subset B_1 \subset B_2 \subset \cdots$ is given by the sequence of projections
\begin{eqnarray*}
  1 + \varpi k^{\circ}/\varpi^2 k^{\circ} \twoheadleftarrow 1 + \varpi k^{\circ}/\varpi^3 k^{\circ} \twoheadleftarrow \cdots \twoheadleftarrow 1 + \varpi k^{\circ} = 1 + k^{\circ \circ}.
\end{eqnarray*}
Thus the system of partitions associated to $A_0 \subset A_1 \subset A_2 \subset \cdots$ is the same one.

\vspace{0.4in}
\addcontentsline{toc}{section}{Acknowledgements}
\noindent {\Large \bf Acknowledgements}
\vspace{0.1in}

I would like to express my deepest gratitude to Takeshi Tsuji for his all instructive advices. I am grateful to Peter Schneider and my colleagues in the University of M\"unster for seminars and discussions during my stay. I am grateful to my friends in the University of Tokyo for their frequent helps. I am thankful to my family for their deep affection.

\end{document}